\documentclass[babel-main=english,authorversion]{TEMat-article}

\usepackage{algpseudocode}
\usepackage{standalone}
\usepackage{tikz}
\usetikzlibrary{calc,positioning,decorations.markings,arrows.meta,backgrounds,patterns,shadows}
\newif\ifbw
\definecolor{orange}{HTML}{f6a800} 
\definecolor{red}{HTML}{bd0000} 
\definecolor{reddot}{HTML}{bd0000} 
\definecolor{green}{HTML}{74DF00} 
\definecolor{gray}{HTML}{646567} 
\definecolor{darkblue}{HTML}{00549F}
\definecolor{lightblue}{HTML}{e8f1fa}

\definecolor{gray2}{HTML}{9c9e9f} 
\definecolor{lightblue2}{HTML}{8ebae5}
\definecolor{petrol}{HTML}{2d7f83}
\definecolor{lila}{HTML}{6A0888}

\definecolor{backgroundcolor}{HTML}{FFFFFF}
\definecolor{backgroundcolor1}{HTML}{e8f1fa}
\definecolor{backgroundcolor2}{HTML}{dcdddd}
\definecolor{backgroundcolor3}{HTML}{eceded}

\tikzstyle{greenline}=[color=green]
\tikzstyle{orangeline}=[color=orange]
\tikzstyle{redline}=[color=red]
\tikzstyle{whiteline}=[color=white]

\tikzstyle{orangedot}=[fill=orange]

\ifbw 

\tikzstyle{green}=[draw=black,fill=black!50]
\tikzstyle{darkblue}=[draw=black,fill=white]
\tikzstyle{orange}=[draw=black,fill=black!85, preaction={fill, white}]
\tikzstyle{gray}=[draw=black,fill=black!20]
\tikzstyle{red}=[draw=black,fill=black!85]

\tikzstyle{petrol}=[fill=black,draw=white,preaction={draw=black,line width=3pt}]
\tikzstyle{lila}=[fill=white,draw=black,copy shadow={shadow scale=1.5,shadow xshift=0pt, shadow yshift=0pt}]
\tikzstyle{lightblue2}=[fill=black!30,draw=black,preaction={draw=black!30,line width=3pt}]
\tikzstyle{gray2}=[fill=black!20,draw=black!60,preaction={draw=black!60,line width=3pt}]

\tikzstyle{orangedot}=[fill=black!50,draw=white,preaction={draw=black!50,line width=3pt}]

\tikzstyle{greenline}=[draw=black!20,dashed]
\tikzstyle{orangeline}=[draw=black!20,dotted]
\tikzstyle{redline}=[draw=black!20,solid]

\fi

\newcommand{\Z}{\mathbb{Z}}
\newcommand{\N}{\mathbb{Z}_+}
\newcommand{\NN}{\mathbb{Z}_{\ge0}}

\newcommand{\R}{\mathbb{R}}

\newcommand{\HH}{\mathcal{H}}
\newcommand{\OO}{\mathcal{O}}

\newcommand{\SL}{\operatorname{SL}}
\newcommand{\PSL}{\operatorname{PSL}}

\newcommand{\ord}{\operatorname{ord}}
\newcommand\inv{^{-1}}

\title[Regular origamis with totally non-congruence groups as Veech groups]{Regular origamis with totally non-congruence groups as Veech groups}

\author*{Thevis, Andrea}
\email{thevis@art.rwth-aachen.de}
\affiliation{Saarland University and RWTH Aachen}


\acknowledgements{This work was funded by the Deutsche Forschungsgemeinschaft (DFG, German Research Foundation) – Project-ID 286237555 – TRR 195 as well as the German Academic Scholarship Foundation. I would like to thank Gabriela Weitze-Schmithüsen for many fruitful discussions.}

\begin{document}

\begin{abstract}
Veech groups are an important tool to examine translation surfaces and related mathematical objects. Origamis, also known as square-tiled surfaces, form an interesting class of translation surfaces with finite index subgroups of $\SL(2,\Z)$ as Veech groups. We study when Veech groups of origamis with maximal symmetry group are totally non-congruence groups, i.e., when they surject onto $\SL(2, \Z/n\Z)$ for each $n\in \N$. For this, we use a result of Schlage-Puchta and Weitze-Schmithüsen to deduce sufficient conditions on the deck transformation group of the origami. More precisely, we show that origamis with certain quotients of triangle groups as deck transformation groups satisfy this condition. All Hurwitz groups are such quotients.
\end{abstract}

\begin{abstract}[spanish]
Los grupos de Veech son una herramienta importante para examinar  superficies de traslación y objetos matemáticos relacionados. Los  origamis, también conocidos como superficies cuadradas, forman una  clase interesante de superficies de traslación con subgrupos de índice  finito en $\SL(2,\Z)$ como grupos de Veech. Estudiamos cuándo los  grupos de Veech de los origamis con grupo de simetría máximo son  grupos totalmente no congruentes, es decir, cuándo surgen en $\SL(2,  \Z/n\Z)$ para cada $n\in \N$. Para ello, utilizamos un resultado de  Schlage-Puchta y Weitze-Schmithüsen para deducir condiciones  suficientes sobre el grupo de transformación de  \extranj{english}{deck} del origami. Más concretamente, mostramos que  los origamis con ciertos cocientes de grupos triangulares como grupos  de transformación de \extranj{english}{deck} satisfacen esta  condición. Todos los grupos de Hurwitz son tales cocientes.
\end{abstract}

\maketitle

\section{Introduction}

A \textbf{translation surface} is a closed Riemann surface with an additional structure which can be described by certain gluing data. We construct such a surface as finitely many polygons in the Euclidean plane with edge identifications along pairs of parallel edges. If all polygons are unit squares one obtains an \textbf{origami} (also known as \textbf{square-tiled surface}). Each origami naturally defines a torus cover sending each square in the tiling to the torus. We are interested in the case where this cover is normal. Then the symmetry group of the origami is maximal and we call the origami \textbf{regular}. A regular origami is completely determined by its deck transformation group $G$ and two deck transformations $x$ and $y$ mapping a fixed square to its right and upper neighbor, respectively (see e.g. \cite{flake2020strata}). We denote such an origami by the tuple $(G,x,y)$.

The matrix group $\SL(2,\R)$ acts on translation surfaces by sheering the polygons in the Euclidean plane. Sometimes the orbit defines an algebraic curve in the moduli space of complex algebraic curves called \textbf{Teichmüller curve}. The stabilizer under this action captures whether this happens and - assuming a positive answer - much of the geometry of the Teichmüller curve. For a translation surface $X$, the stabilizer is called the \textbf{Veech group} of $X$ and is denoted by $\SL(X)$. Origamis define always Teichmüller curves. The Veech groups of \textbf{reduced origamis}, i.e., the geodesic segments between singularities span $\Z^2$, are finite index subgroups of $\SL(2,\Z)$. Here, a singularity means a vertex of a square in the tiling with cone angle larger than $2\pi$. Since non-trivial regular orgiamis are reduced, we restrict ourselves to studying the $\SL(2,\Z)$-action. On a regular origami $\OO=(G,x,y)$ this action is defined by $S\cdot \OO=(G,y^{-1},x)$ and $T\cdot \OO=(G,x,yx^{-1})$, where $S=\left(\begin{smallmatrix}0&-1\\1&0\end{smallmatrix}\right),~ T=\left(\begin{smallmatrix}1&1\\0&1\end{smallmatrix}\right).$ For more details, see e.g. \cite{schmithuesen2004algorithm} and \cite{WS-deficiency}.

We are interested in the following open questions: for which origamis are the Veech groups congruence subgroups and for which are they far away from being a congruence subgroup? Weitze-Schmithüsen showed in \cite{schmithuesen2004algorithm} that almost all congruence groups occur as Veech groups. However, Hubert and Leliévre proved that in the stratum $\HH(2)$ all but one of the occurring Veech groups are not congruence groups (see \cite{HubertLelievre06}). In \cite{WS-deficiency}, Weitze-Schmithüsen introduced the deficiency of finite index subgroups of $\SL(2,\Z)$. It measures how far the group is from being a congruence subgroup. She also established the notion of totally non-congruence groups. Such a group projects surjectively onto $\SL(2,\Z/n\Z)$ for each $n\in \N$, i.e., no information about the group itself can be recovered from the images under these natural projections. In \cite{sp-ws-2018}, an infinite family of origamis with totally non-congruence subgroups as Veech groups are constructed for each stratum. These origamis had only few symmetries. In this article, we present sufficient conditions for regular origamis to have totally non-congruence subgroups as Veech groups and introduce a class of regular origamis satisfying this condition. 

A version of this paper will be published in the proceedings of the 3rd BYMAT Conference 2020.

\section{Prerequisites and preliminary results}\label{sec prerequisites}

In this section, we introduce basic concepts and preliminary results, which are used in \Cref{sec application to regular origamis}. Note that the Euclidean metric on $\R^2$ lifts to a metric on a translation surface. Therefore, notions as directions and geodesics are well-defined on translation surfaces. A \textbf{cylinder} on a translation surface is a maximal collection of parallel closed geodesics. Given a cylinder on a translation surface there exist $w,h>0$ such that the cylinder is isometric to a Euclidean cylinder $\R/w\Z\times (0,h)$. One calls $w$ the \textbf{circumference}, $h$ the \textbf{height}, and the quotient $\frac{h}{w}$ the \textbf{modulus} of the cylinder. If the genus of the translation surface is larger than one, a cylinder is bounded by geodesics between singularities. We call such a geodesic a \textbf{saddle connection}. The direction of a saddle connection bounding a cylinder is called the \textbf{direction of the cylinder}. A \textbf{cylinder decomposition} is a collection of pairwise disjoint cylinders such that the union of their closures covers the whole surface.

The cylinder decomposition of an origami leads to a parabolic element in its Veech group given the following situation. Let $\OO$ be an origami, $v\in \Z^2$ be a rational direction, and $A\in \SL(2,\Z)$ be a matrix mapping $e_1 = \tbinom{1}{0}$ to $v$. If $\OO$ decomposes into cylinders $C_1,\dots,C_k$ with inverse moduli $m_1,\dots,m_k$ and $m$ is the smallest common integer multiple of all the $m_i$, then the matrix 
$A\cdot
\bigl(\begin{smallmatrix}
    1 & m\\
    0 & 1
\end{smallmatrix}\bigr)
\cdot A^{-1}$ is contained in the Veech group $\SL(\OO)$ (see e.g. \cite[Section 2.1]{WS-deficiency}).

\begin{center}
\begin{figure}
  \begin{minipage}[c]{0.27\textwidth}
    \hspace{0.45cm}\begin{tikzpicture}[remember picture, scale=0.8]


\draw[fill=green!40](-1,0)--(3,0)--(3,1)--(-1,1);

\draw[fill=darkblue!30](-1,1)--(0,1)--(0,2)--(-1,2);
\draw[fill=darkblue!30](1,1)--(2,1)--(2,2)--(1,2);
\draw[fill=darkblue!30](0,-1)--(1,-1)--(1,0)--(0,0);
\draw[fill=darkblue!30](2,-1)--(3,-1)--(3,0)--(2,0);




\draw[line width=.25mm] (0,0) -- (1,0) -- (1,1) -- (0,1) -- (0,0);
\draw[line width=.25mm] (1,0) -- (2,0) -- (2,1) -- (1,1);
\draw[line width=.25mm] (2,0) -- (3,0) -- (3,1) -- (2,1);
\draw[line width=.25mm] (0,0) -- (-1,0) -- (-1,1) -- (0,1);

\draw[line width=.25mm] (0,0) -- (0,-1) -- (1,-1) -- (1,0);
\draw[line width=.25mm] (2,0) -- (2,-1) -- (3,-1) -- (3,0);

\draw[line width=.25mm] (1,1) -- (1,2) -- (2,2) -- (2,1);
\draw[line width=.25mm] (0,1) -- (0,2) -- (-1,2) -- (-1,1);

\fill[darkblue] (0,0) circle (0.1);
\fill[darkblue] (2,0) circle (0.1);
\fill[darkblue] (2,2) circle (0.1);
\fill[darkblue] (0,2) circle (0.1);

\fill[orange] (1,0) circle (0.1);
\fill[orange] (3,0) circle (0.1);
\fill[orange] (1,2) circle (0.1);
\fill[orange] (-1,2) circle (0.1);
\fill[orange] (-1,0) circle (0.1);

\fill[gray] (0,1) circle (0.1);
\fill[gray] (2,1) circle (0.1);
\fill[gray] (0,-1) circle (0.1);
\fill[gray] (2,-1) circle (0.1);

\fill[red] (1,-1) circle (0.1);
\fill[red] (3,-1) circle (0.1);
\fill[red] (1,1) circle (0.1);
\fill[red] (3,1) circle (0.1);
\fill[red] (-1,1) circle (0.1);


\draw[line width=.25mm] (-1.1,0.5) -- (-0.9,0.5);
\draw[line width=.25mm] (2.9,0.5) -- (3.1,0.5);

\draw[line width=.25mm] (1.9,-0.45) -- (2.1,-0.45);
\draw[line width=.25mm] (1.9,-0.55) -- (2.1,-0.55);

\draw[line width=.25mm] (-.1,1.45) -- (.1,1.45);
\draw[line width=.25mm] (-.1,1.55) -- (.1,1.55);

\draw[line width=.25mm] (-0.1,-0.5) -- (.1,-0.5);
\draw[line width=.25mm] (-0.1,-0.4) -- (.1,-0.4);
\draw[line width=.25mm] (-0.1,-0.6) -- (.1,-0.6);
\draw[line width=.25mm] (1.9,1.5) -- (2.1,1.5);
\draw[line width=.25mm] (1.9,1.4) -- (2.1,1.4);
\draw[line width=.25mm] (1.9,1.6) -- (2.1,1.6);

\draw[line width=.25mm] (.9,1.6) -- (1.1,1.5) -- (.9,1.4);

\draw[line width=.25mm] (2.9,-.6) -- (3.1,-.5) -- (2.9,-.4);

\draw[line width=.25mm] (-.9,1.6) -- (-1.1,1.5) -- (-.9,1.4);

\draw[line width=.25mm] (1.1,-.6) -- (.9,-.5) -- (1.1,-.4);


\draw[line width=.25mm] (-0.55,1.9) -- (-0.45,2.1);
\draw[line width=.25mm] (-.45,.1) -- (-.55,-.1);

\draw[line width=.25mm] (0.4,-1.1) -- (.5,-.9);
\draw[line width=.25mm] (0.5,-1.1) -- (.6,-.9);

\draw[line width=.25mm] (.4,.9) -- (.5,1.1);
\draw[line width=.25mm] (.5,.9) -- (.6,1.1);

\draw[line width=.25mm] (1.45,2.1) -- (1.55,1.9);
\draw[line width=.25mm] (1.45,.1) -- (1.55,-.1);

\draw[line width=.25mm] (2.4,1.1) -- (2.5,.9);
\draw[line width=.25mm] (2.5,1.1) -- (2.6,.9);

\draw[line width=.25mm] (2.4,-.9) -- (2.5,-1.1);
\draw[line width=.25mm] (2.5,-.9) -- (2.6,-1.1);
\end{tikzpicture}
  \end{minipage}\hfill
  \begin{minipage}[c]{0.72\textwidth}\vspace{0.35cm}
    \caption{Let $D_8$ denote the dihedral group $\langle r,s ~|~ r^{4}=s^{2}=1,~s\inv rs=r\inv  \rangle.$ The clyinder decomposition of the $2$-origami $\OO=(D_8,r,s)$ in horizontal direction consists of two cylinders shaded in green and blue, respectively. The inverse modulus of both cylinders is $4$. Choosing $A$ as the identity element, we obtain that the parabolic matrix $\bigl(\begin{smallmatrix}
    1 & 4\\
    0 & 1
\end{smallmatrix}\bigr)$ lies in the Veech group $\SL(\OO)$.
    } \label{fig::example-cylinderdecomp}
  \end{minipage}
\end{figure}
\end{center}

We conclude this section giving a sufficient condition when finite index subgroups of $\SL(2, \Z)$ are totally non-congruence groups. It is used in the proofs of \Cref{cor. cond. totally NCGp} and \Cref{(abc)-gps totally NCG}.
    
\begin{thm}[{\cite[Theorem 1]{sp-ws-2018}}]\label{thm sp-ws tot NCG}
 Let $\Gamma$ be a finite index subgroup of $\SL(2, \Z)$. Suppose that for each prime $p$ there exist matrices $A_1 , A_2 \in \SL(2, \Z)$ with the following properties:
 \begin{enumerate}
     \item $\forall j \in \N : A_1 e_1 \neq j \cdot A_2 e_1$ modulo $p$.
     \item There exist $m_1 , m_2 \in \N$ with $A_1T^{m_1}A_1^{-1},~A_2T^{m_2}A_2^{-1}\in \Gamma$ such that $p$ divides neither $m_1$ nor $m_2$.
 \end{enumerate}
Then $\Gamma$ is a totally non-congruence group.
\end{thm}

\section{Application to regular origamis}\label{sec application to regular origamis}

In this section, we use cylinder decompositions in different directions to construct the parabolic matrices occurring in \Cref{thm sp-ws tot NCG}. The moduli of the considered cylinders coincide with the order of certain deck transformations. We apply \Cref{thm sp-ws tot NCG} to regular origamis with deck transformation groups that admit for each prime number a cylinder decomposition with a suitable modulus. The following lemma computes the inverse moduli of the cylinders in the directions of interest.

\begin{lemma}\label{lem moduli of cylinders}
    Let $\OO=(G,x,y)$ be a regular origami. For $m\in \NN$, the inverse modulus of all cylinders in direction $\tbinom{1}{-m}$ coincides with the order of $xy^m$.
\end{lemma}

\begin{proof}
    Denote $\tbinom{1}{-m}$ by $v$. Acting with the matrix 
    $A=\bigl(
    \begin{smallmatrix}
    1 & 0\\
    -m & 1
    \end{smallmatrix}
    \bigr)=\left( S^{3}TS\right)^m\in \SL(2,\Z)$ maps the horizontal direction to the direction $v$, i.e., $A\cdot e_1=v$. The inverse modulus of all horizontal cylinders of the origami $A\cdot\OO=(G,xy^m,y)$ coincides with the order of $xy^m$. Note that acting by matrices in $\SL(2,\Z)$ does not change the modulus of a cylinder. Hence, the inverse modulus of the cylinder in direction $v$ of the origami $\OO$ equals the order of $xy^m$. 
\end{proof}

Using \Cref{thm sp-ws tot NCG} and \Cref{lem moduli of cylinders}, we deduce a sufficient condition for regular origamis to have a totally non-congruence group as Veech group.

\begin{proposition}\label{cor. cond. totally NCGp}
    Let $\OO=(G,x,y)$ be a regular origami. If for each prime $p$ one of the following holds 
    \begin{enumerate}
        \item $\gcd(p,\ord(y)\cdot\ord(yx))=1$ or
        \item there exist $m_1,m_2\in \NN$ with $\overline{m}_1\not\equiv \overline{m}_2 \mod p$ and $\gcd(p,\ord(xy^{-m_1})\cdot\ord(xy^{-m_2}))=1$.
    \end{enumerate}
    Then the Veech group $\SL(\OO)$ is a totally non-congruence group.
\end{proposition}

\begin{proof}
    Fix a prime $p$. If condition (1) holds, consider the matrices $S^{-1}T^{-1}=\bigl(
    \begin{smallmatrix}
    0 & 1\\
    -1 & 1
    \end{smallmatrix}
    \bigr)$ and $TS^{-1}=\bigl(
    \begin{smallmatrix}
    -1 & 1\\
    -1 & 0
    \end{smallmatrix}
    \bigr).$  We obtain $S^{-1}T^{-1}\cdot\OO=(G,yx,x^{-1})$ and $TS^{-1}\cdot\OO=(G,y,x^{-1}y^{-1})$. The moduli of the horizontal cylinders of the regular origamis $(G,yx,x^{-1})$ and $(G,y,(yx)^{-1})$ are $\ord(yx)=\vcentcolon a$ and $\ord(y)=\vcentcolon b$, respectively. Hence, $S^{-1}T^{-1} \cdot T^a\cdot TS$ and $TS^{-1}\cdot T^b \cdot ST^{-1}$ lie in the Veech group $\SL(\OO)$. Moreover, we obtain for each $j\in\Z$ the inequality $S^{-1}T^{-1}\cdot e_1=\tbinom{0}{-1}\neq j\cdot \tbinom{-1}{-1}= j\cdot TS^{-1}\cdot e_1$ modulo $p$. 
    
    If condition (2) holds, then let $m_1,m_2$ be natural numbers satisfying condition (2). Define the matrices 
    $A_1=\bigl(
    \begin{smallmatrix}
    1 & 0\\
    m_1 & 1
    \end{smallmatrix}
    \bigr),~ 
    A_2=\bigl(
    \begin{smallmatrix}
    1 & 0\\
    m_2 & 1
    \end{smallmatrix}
    \bigr).$ Since ${m}_1\not\equiv {m}_2 \mod p$, we have $A_1 e_1 \neq j \cdot A_2 e_1$ modulo $p$ for each $j\in\N$.
    
    As $\gcd(p,\ord(xy^{-m_1})\cdot\ord(xy^{-m_2}))=1$, set $k_1=\ord(xy^{-m_1})$ and $k_2=\ord(xy^{-m_2})$. Using \Cref{lem moduli of cylinders}, we conclude that the matrices $A_iT^{k_i}A_i^{-1}$ are contained in the Veech group of the origami $\OO$. By \Cref{thm sp-ws tot NCG}, the claim follows.
\end{proof}

In the following corollary, we construct generating sets $\{x,y\}$ of alternating groups $A_n$ satisfying the conditions given in \Cref{cor. cond. totally NCGp}. Consequently, the infinite family of regular origamis $(A_n,x,y)$ have totally non-congruence groups as Veech groups.

\begin{corollary}\label{A_n totally NCG}
    For each prime $n\ge 5$, the regular origami $(A_n, (1,2,3), (1,2,3,\dots,n))$ has a totally non-congruence group as Veech group.
\end{corollary}

\begin{proof}
    Set $x\vcentcolon=(1,2,3)$ and $y\vcentcolon=(1,2,3,\dots,n)$. For each prime $p\neq n$, we consider the group elements $y$ and $yx$. Since the orders of $y$ and $yx$ are equal to $n$, the prime $p$ does not divide $\ord(y)\cdot\ord(yx).$
    
    For the prime $n$, we consider the group elements $xy^{n-1}$ and $x$, i.e., $m_1=1-n$ and $m_2=0$. Note that ${1-n}\not\equiv 0\mod n$. The permutation $xy^{n-1}$ has the fixed point $2$ and thus $n$ does not divide the order of $xy^{n-1}$. Since $\ord(x)=3<n$, the prime $n$ does not divide the order of $x$ either. By \Cref{cor. cond. totally NCGp}, the claim follows.
\end{proof}

\begin{example}
    We consider the regular origami $\OO\vcentcolon=(A_5, (1,2,3), (1,2,3,4,5))$ given in \Cref{A_n totally NCG} for the prime $n=5$. 
    \begin{center}
    \captionsetup{type=figure}
    \begin{tikzpicture}[remember picture, scale=0.84]



\draw (0,-9) -- (1,-9) -- (1,-8) -- (0,-8) -- (0,-9);

\draw (0,-8) -- (1,-8) -- (1,-7) -- (0,-7) -- (0,-8);
\draw (1,-8) -- (2,-8) -- (2,-7) -- (1,-7);
\draw (2,-8) -- (3,-8) -- (3,-7) -- (2,-7);

\draw (1,-7) -- (2,-7) -- (2,-6) -- (1,-6) -- (1,-7);

\draw (3,-6) -- (4,-6) -- (4,-5) -- (3,-5);
\draw (2,-6) -- (3,-6) -- (3,-5) -- (2,-5);
\draw (1,-6) -- (2,-6) -- (2,-5) -- (1,-5) -- (1,-6);

\draw (3,-5) -- (4,-5) -- (4,-4) -- (3,-4) -- (3,-5);
\draw (1,-5) -- (2,-5) -- (2,-4) -- (1,-4) -- (1,-5);

\draw (3,-4) -- (4,-4) -- (4,-3) -- (3,-3) -- (3,-4);
\draw (1,-4) -- (2,-4) -- (2,-3) -- (1,-3) -- (1,-4);
\draw (10,-4) -- (11,-4) -- (11,-3) -- (10,-3) -- (10,-4);

\draw (3,-3) -- (4,-3) -- (4,-2) -- (3,-2) -- (3,-3);
\draw (8,-3) -- (9,-3) -- (9,-2) -- (8,-2) -- (8,-3);
\draw (9,-3) -- (10,-3) -- (10,-2) -- (9,-2);
\draw (10,-3) -- (11,-3) -- (11,-2) -- (10,-2);

\draw (0,-2) -- (1,-2) -- (1,-1) -- (0,-1) -- (0,-2);
\draw (2,-2) -- (3,-2) -- (3,-1) -- (2,-1) -- (2,-2);
\draw (3,-2) -- (4,-2) -- (4,-1) -- (3,-1);
\draw (4,-2) -- (5,-2) -- (5,-1) -- (4,-1);
\draw (10,-2) -- (11,-2) -- (11,-1) -- (10,-1) -- (10,-2);
\draw (11,-2) -- (12,-2) -- (12,-1) -- (11,-1);

\draw (-1,-1) -- (0,-1) -- (0,0) -- (-1,0) -- (-1,-1);
\draw (0,-1) -- (1,-1) -- (1,0) -- (0,0);
\draw (2,-1) -- (3,-1) -- (3,0) -- (2,0) -- (2,-1);
\draw (6,-1) -- (7,-1) -- (7,0) -- (6,0) -- (6,-1);
\draw (10,-1) -- (11,-1) -- (11,0) -- (10,0) -- (10,-1);

\draw (0,0) -- (1,0) -- (1,1) -- (0,1) -- (0,0);
\draw (1,0) -- (2,0) -- (2,1) -- (1,1);
\draw (2,0) -- (3,0) -- (3,1) -- (2,1);
\draw (6,0) -- (7,0) -- (7,1) -- (6,1) -- (6,0);
\draw (8,0) -- (9,0) -- (9,1) -- (8,1) -- (8,0);
\draw (9,0) -- (10,0) -- (10,1) -- (9,1);
\draw (10,0) -- (11,0) -- (11,1) -- (10,1);

\draw (0,1) -- (1,1) -- (1,2) -- (0,2) -- (0,1);
\draw (2,1) -- (3,1) -- (3,2) -- (2,2) -- (2,1);
\draw (4,1) -- (5,1) -- (5,2) -- (4,2) -- (4,1);
\draw (5,1) -- (6,1) -- (6,2) -- (5,2);
\draw (6,1) -- (7,1) -- (7,2) -- (6,2);
\draw (8,1) -- (9,1) -- (9,2) -- (8,2) -- (8,1);

\draw (2,2) -- (3,2) -- (3,3) -- (2,3) -- (2,2);
\draw (3,2) -- (4,2) -- (4,3) -- (3,3);
\draw (4,2) -- (5,2) -- (5,3) -- (4,3);
\draw (6,2) -- (7,2) -- (7,3) -- (6,3) -- (6,2);
\draw (8,2) -- (9,2) -- (9,3) -- (8,3) -- (8,2);

\draw (4,3) -- (5,3) -- (5,4) -- (4,4) -- (4,3);
\draw (6,3) -- (7,3) -- (7,4) -- (6,4) -- (6,3);
\draw (7,3) -- (8,3) -- (8,4) -- (7,4);
\draw (8,3) -- (9,3) -- (9,4) -- (8,4);

\draw (4,4) -- (5,4) -- (5,5) -- (4,5);
\draw (3,4) -- (4,4) -- (4,5) -- (3,5);
\draw (2,4) -- (3,4) -- (3,5) -- (2,5) -- (2,4);

\draw (4,5) -- (5,5) -- (5,6) -- (4,6) -- (4,5);
\draw (5,5) -- (6,5) -- (6,6) -- (5,6);
\draw (2,5) -- (3,5) -- (3,6) -- (2,6) -- (2,5);

\draw (5,6) -- (6,6) -- (6,7) -- (5,7) -- (5,6);

\draw (5,7) -- (6,7) -- (6,8) -- (5,8) -- (5,7);

\draw (5,8) -- (6,8) -- (6,9) -- (5,9) -- (5,8);

\draw (5,9) -- (6,9) -- (6,10) -- (5,10) -- (5,9);

\def\nextflag#1{%
\pgfmathsetmacro{\test}{isodd(\flags)}%
\ifnum\test=1 \draw #1; \fi%
\pgfmathsetmacro\flags{div(\flags,2)}%
}

\def\mymark#1#2#3{
\begin{scope}[xshift=#2cm,yshift=#3cm]
\pgfmathsetmacro\flags{#1}
\nextflag{(-.1,0) -- (.1,0)} 
\nextflag{(-.1,.1) -- (.1,.1)} 
\nextflag{(-.1,-.1) -- (.1,-.1)} 
\nextflag{(0,0) circle (.1)} 
\nextflag{(-.1,.1) -- (.1,-.1)} 
\nextflag{(-.1,-.1) -- (.1,.1)} 
\nextflag{(-.1,.1) -- (.1,0) -- (-.1,-.1)} 
\nextflag{(.1,.1) -- (-.1,0) -- (.1,-.1)} 
\nextflag{(-.1,-.1) -- (-.1,.1)} 
\nextflag{(.1,-.1) -- (.1,.1)} 
\end{scope}
}
\mymark{1}{3}{-.5} \mymark{1}{4}{5.5}
\mymark{6}{6}{5.5} \mymark{6}{2}{-0.5}

\mymark{7}{6}{6.5} \mymark{7}{1}{-3.5}
\mymark{8}{2}{-3.5} \mymark{8}{3}{-2.5}
\mymark{16}{4}{-2.5} \mymark{16}{5}{6.5}

\mymark{18}{12}{-1.5} \mymark{18}{8}{2.5}
\mymark{20}{9}{2.5} \mymark{20}{10}{-1.5}
\mymark{32}{1}{-.5} \mymark{32}{3}{-3.5}
\mymark{34}{4}{-3.5} \mymark{34}{-1}{-.5}

\mymark{36}{11}{-.5} \mymark{36}{2}{5.5}
\mymark{48}{3}{5.5} \mymark{48}{0}{-1.5}
\mymark{50}{1}{-1.5} \mymark{50}{10}{-.5}

\mymark{52}{6}{9.5} \mymark{52}{6}{2.5}
\mymark{64}{7}{2.5} \mymark{64}{1}{-6.5}
\mymark{66}{2}{-6.5} \mymark{66}{5}{9.5}

\mymark{68}{7}{.5} \mymark{68}{8}{1.5}
\mymark{128}{9}{1.5} \mymark{128}{1}{-4.5}
\mymark{130}{2}{-4.5} \mymark{130}{6}{.5}

\mymark{132}{6}{7.5} \mymark{132}{4}{3.5}
\mymark{257}{5}{3.5} \mymark{257}{6}{-.5}
\mymark{262}{7}{-.5} \mymark{262}{5}{7.5}

\mymark{263}{1}{-8.5} \mymark{263}{10}{-3.5}
\mymark{513}{11}{-3.5} \mymark{513}{0}{1.5}
\mymark{518}{1}{1.5} \mymark{518}{0}{-8.5}

\mymark{519}{3}{1.5} \mymark{519}{5}{8.5}
\mymark{24}{6}{8.5} \mymark{24}{3}{-4.5}
\mymark{14}{4}{-4.5} \mymark{14}{2}{1.5}

\mymark{2+4+16+32}{11.5}{-1} \mymark{2+4+16+32}{2.5}{-6}
\mymark{768}{2.5}{-5} \mymark{768}{.5}{-9}
\mymark{2+4+16+32+256+512}{.5}{-7} \mymark{2+4+16+32+256+512}{-0.5}{-1}
\mymark{816}{-.5}{0} \mymark{816}{11.5}{-2}

\mymark{64+256}{8.5}{0} \mymark{64+256}{2.5}{-7}
\mymark{64+512}{8.5}{4} \mymark{64+512}{2.5}{-8}

\mymark{128+256}{.5}{-2} \mymark{128+256}{4.5}{-1}
\mymark{128+512}{.5}{2} \mymark{128+512}{4.5}{-2}

\mymark{16+32+8}{2.5}{4} \mymark{16+32+8}{1.5}{1}
\mymark{519+256}{1.5}{0} \mymark{519+256}{9.5}{-2}
\mymark{256}{3.6}{3} \mymark{256}{9.6}{-3}
\mymark{768}{2.5}{6} \mymark{768}{3.5}{2} \mymark{256}{2.6}{6} \mymark{256}{3.6}{2}

\mymark{16}{3.6}{4} \mymark{16}{3.4}{4}
\mymark{16}{9.6}{1} \mymark{16}{9.4}{1}
\mymark{32}{9.6}{0} \mymark{32}{9.4}{0}
\mymark{32}{5.6}{2} \mymark{32}{5.4}{2}
\mymark{8}{5.5}{1} \mymark{256}{5.6}{1}
\mymark{8}{8.5}{-2} \mymark{256}{8.6}{-2}
\mymark{256}{8.6}{-3} \mymark{48}{8.5}{-3}
\mymark{256}{7.6}{4} \mymark{48}{7.5}{4}
\mymark{256}{7.6}{3} \mymark{6}{7.5}{3}
\mymark{256}{3.6}{5} \mymark{6}{3.5}{5}

\end{tikzpicture}
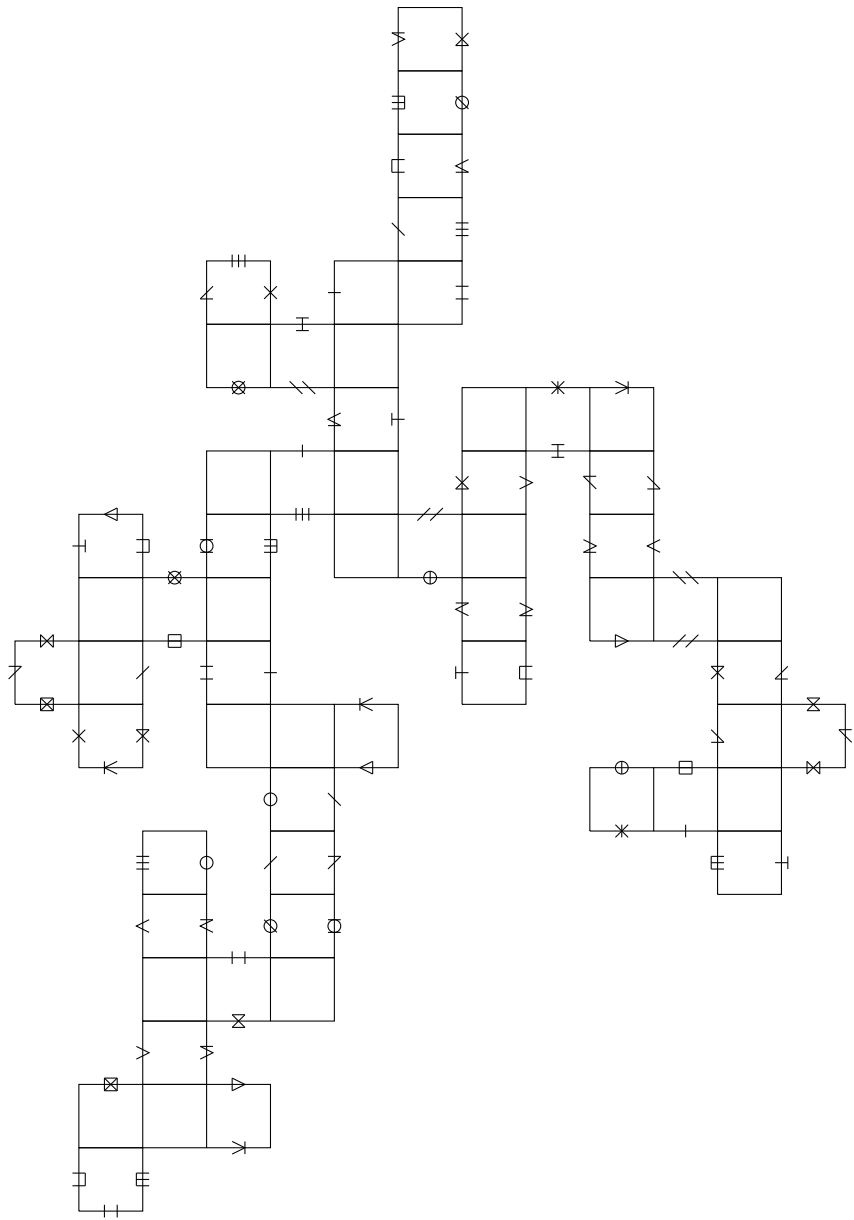
    \caption{Each unlabeled edge is glued to the opposite edge. For clarity, these edges are not labeled.}
    \label{fig::example-A5}
\end{center}
Using the \texttt{GAP} package \cite{OrigamiPackage}, we compute that the Veech group $\SL(\OO)$ is generated by the matrices
$$S^2, TST^{-1}, T^3,T^{-1}STS^{-1},STST^{-3}S^{-1}$$
and has index $9$ in $\SL(2,\Z)$. Representatives of the $\SL(2,\Z)$-orbit are given by the following origamis 
\begin{align*}
    &\OO = (A_5, (1,2,3), (1,2,3,4,5)),\hphantom{\vcentcolon 4_3.,5} \quad\OO_2\vcentcolon=(A_5, (2,4)(3,5), (1,2,3,4,5)),\\
    &\OO_3\vcentcolon=(A_5, (1,2,4,5,3), (1,2,3,5,4)),\quad~\OO_4\vcentcolon=(A_5, (3,5,4), (1,2,3,4,5)),\\
    &\OO_5\vcentcolon=(A_5, (1,3,2,5,4), (1,2)(3,4)),\hphantom{5}\quad ~\OO_6\vcentcolon=(A_5, (1,2,3,4,5), (1,2,3)),\\
    &\OO_7\vcentcolon=(A_5,(1,3,5,4,2), (1,2,3)), \hphantom{4),5}\quad~\OO_8\vcentcolon=(A_5, (1,2,3,4,5), (1,2,3,5,4)), \\
    &\OO_9\vcentcolon=(A_5, (3,4,5), (1,2,3)).
\end{align*}
\end{example}

\Cref{A_n totally NCG} motivates to examine finite simple groups more generally. Simple groups form an interesting class of $2$-generated groups. The natural question, how the orders of generators for a fixed group can be chosen, has been studied intensively (see e.g. \cite{JLM} for further information). This question suggests to consider $(a,b,c)$-groups.

\begin{definition}
    A finite group generated by two elements $x,y$ with $\ord(x)=a, \ord(y)=b,$ and $\ord(xy)=c$ is called an \textbf{$(a,b,c)$-group}. We call such generators \textbf{$(a,b,c)$-generators}.
\end{definition}

Each $(a,b,c)$-group is a finite quotient of the \textbf{triangle group} $$T_{(a,b,c)}=\langle x,y,z~|~x^a=y^b=z^c=xyz=1\rangle.$$ The following theorem shows that $(a,b,c)$-groups where $a,b,$ and $c$ are chosen pairwise coprime produce regular origamis with a totally non-congruence group as Veech group.

\begin{theorem}\label{(abc)-gps totally NCG}
    Let $a,b,c\in\NN$ be pairwise coprime and $G$ be an $(a,b,c)$-group with $(a,b,c)$-generators $x,y$. The Veech group of the regular origami $(G,y,x)$  is a totally non-congruence group. 
\end{theorem}

\begin{proof}
    We prove that the assumptions of \Cref{thm sp-ws tot NCG} are satisfied for the Veech group of the regular origami $\OO=(G,y,x)$. Let $p$ be a prime. Since $a,b,$ and $c$ are pairwise coprime, $p$ divides at most one of the numbers $a,b,$ and $c$. We consider each of the three cases separartely.
    
    If $p$ is coprime to $b\cdot c$, then consider the matrices $I$ and $S^{-1}T^{-1}=\bigl(
    \begin{smallmatrix}
    0 & 1\\
    -1 & 1
    \end{smallmatrix}
    \bigr).$ 
    We obtain $I\cdot\OO=\OO$ and $S^{-1}T^{-1}\cdot\OO=(G,xy,y^{-1})$. The inverse moduli of the horizontal cylinders of the regular origamis $\OO$ and $(G,xy,y^{-1})$ are $\ord(y)=b$ and $\ord(xy)=c$, respectively. Hence, $T^b$ and $S^{-1}T^{-1}\cdot T^c \cdot TS$ lie in the Veech group $\SL(\OO)$. Moreover, we obtain $S^{-1}T^{-1}\cdot e_1=\tbinom{0}{-1}\neq j\cdot\tbinom{1}{0}$ modulo $p$ for each $j\in\Z$. 
    
    If $p$ is coprime to $a\cdot c$, then consider the matrices $S^{-1}T^{-1}=\bigl(
    \begin{smallmatrix}
    0 & 1\\
    -1 & 1
    \end{smallmatrix}
    \bigr)$ and $TS^{-1}=\bigl(
    \begin{smallmatrix}
    -1 & 1\\
    -1 & 0
    \end{smallmatrix}
    \bigr).$ We obtain the regular origamis $S^{-1}T^{-1}\cdot\OO=(G,xy,y^{-1})$ and $TS^{-1}\cdot\OO=(G,x,y^{-1}x^{-1})$. The moduli of the horizontal cylinders of the regular origamis $(G,xy,y^{-1})$ and $(G,x,(xy)^{-1})$ are $\ord(xy)=c$ and $\ord(x)=a$, respectively. Hence, $S^{-1}T^{-1} \cdot T^c\cdot TS$ and $TS^{-1}\cdot T^a \cdot ST^{-1}$ lie in the Veech group $\SL(\OO)$. Moreover, we obtain for each $j\in\Z$ the inequality $S^{-1}T^{-1}\cdot e_1=\tbinom{0}{-1}\neq j\cdot \tbinom{-1}{-1}= j\cdot TS^{-1}\cdot e_1$ modulo $p$.
    
    If $p$ is coprime to $a\cdot b$, then consider the matrices $I$ and $S^{-1}=\bigl(
    \begin{smallmatrix}
    0 & 1\\
    -1 & 0
    \end{smallmatrix}
    \bigr)$. We obtain the regular origamis $I\cdot\OO=\OO$ and $S^{-1}\cdot\OO=(G,x,y^{-1})$. The moduli of the horizontal cylinders of the regular origamis $\OO$ and $(G,x,y^{-1})$ are $\ord(y)=b$ and $\ord(x)=a$, respectively. Hence, $T^b$ and $S^{-1}\cdot T^a \cdot S$ lie in the Veech group $\SL(\OO)$. Moreover, we have $S^{-1}\cdot e_1=\tbinom{0}{-1}\neq j\cdot\tbinom{1}{0}$ modulo $p$ for each $j\in\Z$. 
\end{proof}

\begin{example}\label{ex hurwitz groups}
    A well-studied family of groups satisfying the assumption in \Cref{(abc)-gps totally NCG} are $(2,3,7)$-groups, which are also called Hurwitz groups. Hurwitz groups are of interest from a geometric point of view because they arise as automorphism groups of compact Riemann surfaces of genus $g>1$ with minimal order $84(g-1)$. The smallest Hurwitz group is the projective linear group $\PSL(2,7)$ and has order $168$. For further information about Hurwitz, see e.g. \cite{hurwitz1} and \cite{hurwitz2}.
\end{example}

\printbibliography[heading=bibintoc]

\end{document}